\documentclass[a4paper,twoside,11pt]{article}
\usepackage[utf8]{inputenc}
\usepackage[english,russian]{babel}
\usepackage{fancyhdr}
\usepackage{newprog1e}
\usepackage{amsfonts,amsmath,amssymb,amsthm}
\usepackage{graphicx}
\usepackage{xcolor}

\numberwithin{equation}{section}
\newtheorem{theorem}{Теорема}

\journalnumber{4}
\curyear{2023}
\authorlist{Р.\,Ведерников и др.}
\titlehead{ДЕЦЕНТРАЛИЗОВАННЫЙ МЕТОД УСЛОВНОГО ГРАДИЕНТА НА ПЕРЕМЕННЫХ ГРАФАХ}
\headerdef

\udk{519.85}
\rubrika{}
\dateinput{14.07.2023}

\rusabstr{В данной работе рассматривается обобщение децентрализованного алгоритма Франк-Вульфа на переменные во времени сети, исследуются свойства сходимости алгоритма и проводятся соответственные численные эксперименты. Меняющаяся сеть моделируется как детерминированная или стохастическая последовательность графов.}

\author{
{\bfseries Р.\,А.\,Ведерников\,$^{\it{a}\,*}$,
А.\ В.\, Рогозин$^{\it{a}\,**}$,
А.\ В.\, Гасников,$^{\it{b,c}\,***},$
}
\\ {\itshape $^{\it{a}}$\,Московский физико-технический институт}
\\ {\slshape 141701,} {\itshape г. Долгопрудный, Институтский пер.,~9}
\\ {\itshape $^{\it{b}}$\,Институт проблем передачи информации}
\\ {\slshape 127051,} {\itshape г. Москва, Большой Каретный пер., 19, стр. 1}
\\ {\itshape $^{\it{c}}$\,Кавказский математический центр Адыгейского государственного университета}
\\ {\slshape 385000,} {\itshape г. Майкоп, Первомайская ул., 208}
\\ {\itshape *\,e-mail: vedernikov.ra@phystech.edu \\
**\,e-mail: aleksandr.rogozin@phystech.edu \\
***\,e-mail: gasnikov@yandex.ru}}
\title{ДЕЦЕНТРАЛИЗОВАННЫЙ МЕТОД УСЛОВНОГО ГРАДИЕНТА НА ПЕРЕМЕННЫХ ВО ВРЕМЕНИ ГРАФАХ}

\date{}

\usepackage{algorithm}
\usepackage{algorithmic}
\floatname{algorithm}{Алгоритм}
\usepackage{url}
\usepackage{hyperref}

\newtheorem{lemma}{Лемма}

\newtheorem{assumption}{Предположение}
\newtheorem{corollary}{Следствие}
\renewcommand{\phi}{\varphi}
\renewcommand{\epsilon}{\varepsilon}
\renewcommand{\leq}{\leqslant}
\renewcommand{\geq}{\geqslant}
\usepackage{sectsty}
\sectionfont{\MakeUppercase}
\usepackage{titlesec}
\usepackage{caption}
\titlelabel{\thetitle. }

\usepackage{pifont}

\newcommand{\E}{\mathbb{E}}
\newcommand{\R}{\mathbb{R}}

\newcommand{\norm}[1]{\left\| #1 \right\|}

\newcommand{\cbraces}[1]{\left( #1 \right)}
\newcommand{\sbraces}[1]{\left[ #1 \right]}

\newcommand{\circledOne}{\text{\ding{172}}}

\newcommand{\rev}[1]{{\color{black}#1}}

\begin{document}

\maketitle

\section*{Введение}\label{sec1_introduction}

Алгоритм Франк-Вульфа \cite{cgm}, (также известен как метод условного градиента или метод Левитина-Поляка \cite{levitin-polyak}) - итеративный алгоритм оптимизации, который часто используется для решения задач выпуклой оптимизации. Он был представлен Маргаритой Франк и Филиппом Вульфом в 1956 году.

\begin{algorithm}[!ht]
\caption{Классический метод условного градиента}\label{classical_FW}
\begin{algorithmic}[1]
   \REQUIRE Количество итераций $m$, начальная точка $x_{0} \in Q$.
   \FOR{$t=0, 1, \ldots, m-1$}
    \STATE $\alpha_t = \frac{2}{t+1}$
    \STATE $s_t = \arg\min_{x \in Q}\left\{  \nabla f(x_t)^\top x \right\}$
    \STATE $x_{t+1} = (1-\alpha_t) x_t + \alpha_t s_t$
    \ENDFOR
\ENSURE $x_m$
\end{algorithmic}
\end{algorithm}

Основная идея алгоритма Франк-Вульфа заключается в следующем:

Алгоритм инициализируется в пределах допустимого множества $D$. После этого начинаются итерации алгоритма: на каждой итерации мы приближаем целевую функцию линейной функцией в окрестности текущей точки, и ищем точку из допустимого множества, проекция которой на направление антиградиента будет максимальной. Эта точка задает направление шага алгоритма, которое, вообще говоря, может не совпадать с направлением антиградиента, что отличает его от градиентного спуска.

После выбора направления движения, есть два основных способа задать величину шага. Первый - выбрать шаг заранее и задать функцией от номера итерации:
\begin{equation}
    \gamma_t = \dfrac{2}{t+2},
\end{equation}

Второй способ - техника short step rule, которая заключается в решении задачи минимизации функции на выбранном направлении по допустимому множеству на каждой итерации: 
\begin{equation}
    \gamma_t = \underset{\gamma\geq 0}{\text{argmin}}~ f(x_t + \gamma (s_t - x_t)).
\end{equation}

Подзадача линейной минимизации часто проще, чем исходная задача. Алгоритм может быть особенно полезен, когда допустимая область является компактным и выпуклым множеством в пространстве большой размерности. 

В работе рассматривается приложение алгоритма Франк-Вульфа к решению задач на сетях, в силу особенностей топологии не имеющих общего распределяющего центра и требующих применения децентрализованного алгоритма.


\section{Децентрализованная оптимизация}
\subsection{Постановка задачи}

Рассмотрим произвольную систему из $N$ узлов. Узлы могут обмениваться информацией через переменную во времени сеть, каждый узел связан с некоторыми другими, может передавать и получать информацию только от них. Таким образом, систему можно представить последовательностью неориентированных графов  $G^t = (V, E^t)$, причем в ней нет главного узла, который мог бы аггрегировать информацию. Будем требовать, чтобы на каждой итерации граф системы оставался связным. 

Будем рассматривать задачу минимизации суммы функции:
\begin{equation}
    \min_{x\in D} f(x) = \dfrac{1}{N} \sum^N_{i = 1} f_i(x).
\end{equation}

Эта задача относится к задачам децентрализованной оптимизации. Каждый $i-\text{й}$ узел хранит состояние $x$, умеет вычислять свою функцию $f_i(x)$, а так же ее градиент $\nabla f(x)$ в этой точке.

Задачи такого типа имеют множество приложений в областях, где ограничена или невозможна аггрегация информации из-за ограничений безопасности, архитектуры сети или размеров данных, например, в распределенном машинном обучении \cite{Nedic} \cite{distributed_ml}, системах контроля мощностей \cite{charging} \cite{powercontrol}, контроле и управлении техникой \cite{vehicle}.

\subsection{Коммуникационная матрица}

Важную роль в алгоритмах децентрализованной оптимизации играет процесс консенсуса, который реализует обмен информации между узлами.

Коммуникационная матрица \cite{rogozin} - это матрица, используемая в этом процессе, где каждый элемент представляет собой силу связи или вес между двуми узлами. Коммуникационная матрица  является важной частью алгоритма консенсуса, который помогает всем узлам достичь согласия относительно оптимального решения.

В процессе децентрализованной оптимизации алгоритм консенсуса работает следующим образом:

\begin{itemize}
  \item Каждый узел начинает с начальной оценки решения.
  \item Каждый узел сообщает свою оценку своим соседям.
  \item Каждый узел обновляет свою оценку на основе полученной от своих соседей информации, взвешенной согласно коммуникационной матрице.
  
\end{itemize}

Этот процесс повторяется до тех пор, пока все узлы не достигнут консенсуса, то есть их оценки не сойдутся к одному и тому же значению.
Веса в коммуникационной матрице могут быть скорректированы в соответствии с потребностями системы, например, чтобы дать больший вес оценкам узлов, которые известны своей большей точностью или надежностью. 

На последовательность коммуникационных матриц накладываются следующие условия:
\begin{assumption}\label{assum:mixing_matrix_sequence}
	Для каждого $t = 0, 1, \ldots$ выполняется
	\begin{enumerate}
		\item (Согласованность с графом) $[W^t]_{ij} = 0$ if $(i, j)\notin E^t$ и $i\neq j$.
		\item (Дважды стохастичность) $W^t\mathbf{1} = \mathbf{1},~ \mathbf{1}^\top W^t = \mathbf{1}^\top$.
		 \item (Свойство сжатия) Найдется такое $\lambda< 1$, что для любого $t = 0, 1, \ldots$ выполняется
		 \begin{align*}
		 	\norm{\cbraces{W^t - \frac{1}{N}\mathbf{1}\mathbf{1}^\top}x}\leq \lambda\norm{x}.
		 \end{align*}
	\end{enumerate}
\end{assumption}

Построение коммуникационной матрицы такой, чтобы она обладала свойством сжатия, может быть неочевидным. Приведем достаточные условия для выполнения данного свойства. Пусть для любого $t = 0, 1, \ldots$ выполняется:\\
1. $W^t \mathbf{1} = \mathbf{1}, \   \mathbf{1}^\top W^t = \mathbf{1}^\top $. \\
2. Для любого $i = 1, \ldots, N$ выполняется $[W^t]_{ii} > 0$.\\
3. Если $(i,j) \in E^t$, то $[W^t]_{ij} > 0$, иначе  $[W^t]_{ij} = 0$.\\
4. Существует $\theta> 0$, такое что если $ [W^t]_{ij} > 0$, то $[W^t]_{ij} \geq \theta$.

В этой работе мы будем использовать способ выбора весов Metropolis Weights, который имеет следующий вид:

\begin{equation}\label{eq:metropolis_weights}
\left[ W^t \right]_{ij} = 
 \begin{cases}
   \dfrac{1}{\text{max(deg(i), deg(j))} + 1}, 
&\text{$(i,j) \in E^t$,} \\
   0, &\text{$(i,j) \notin E^t,$}\\
   1 - \sum_{i \neq j } \left[ W^t  \right]_{ij}, &\text{$i = j$}.
 \end{cases}
\end{equation}

\subsection{Построение алгоритма}

Децентрализованный алгоритм Франк-Вульфа строится из его классической версии \cite{base}.
Пусть $t \in \mathbb{N}$ - номер итерации, а начальная точка 
$\theta_0 \in D$ взята из допустимого множества. 
Напомним, что $F(x) = \dfrac{1}{N} \sum\limits_{i=1}^N f_i (x^i)$ \rev{(где  $x$ -- матрица, строками которой являются $(x^i)^\top$)},
тогда централизованный алгоритм действует следующим образом:
\begin{gather}
    v_t \in \text{argmin}_{v \in D^N} \langle \nabla F (x_t), v \rangle,  \\
x_{t} = x_{t-1} + \gamma_{t-1} ( v_{t-1} - x_{t-1} ) , 
\end{gather}
где $\gamma_{t-1} \in (0,1]$ - заданная величина шага алгоритма.
Заметим, что $x_{t}$ - выпуклая комбинация $x_{t-1}$ and $v_{t-1}$ которые лежат в допустимом множестве, поэтому также принадлежит допустимому множеству.
Когда шаг алгоритма задается как $\gamma_t = 2/(t+2)$, известна оценка скорости сходимости $O(1/t)$, если $F$ является $L$-гладкой 
и выпуклой функцией.
Следующим этапом будет децентрализация алгоритма. Для этого необходимо заменить централизованные значения функции и градиента на их локальные приближения. Во-первых, определим среднее координат точек:
\begin{equation} 
    \overline{x_t} = \dfrac{1}{N} \sum\limits_{i=1}^N x^i_t
\end{equation}
и среднее градиентов в координатах узлов:

\begin{equation}
    \overline{\nabla_t F} = \dfrac{1}{N} \sum\limits_{i=1}^N \nabla f_i(\overline{x}_t^i).
\end{equation}
Эти величины понадобятся нам, чтобы оценить скорость сходимости алгоритма. Во-вторых,  определим локальные аппроксимации этих величин, доступные для вычисления в каждом узле:

Вычисление локальной аппроксимации точки консенсуса $\overline{x}_t^i $ называется шагом консенсуса (consensus step):
\begin{equation}
    \overline{x}_t^i = \sum\limits_{j=1}^N W^t_{ij} \cdot x^j_t.
\end{equation}
Здесь реализуется связь соседних узлов, причем информация из несвязанных между собой узлов этими узлами игнорируется, т.к. $W^t_{ij} = 0 , \ (i,j) \notin E^t$.

Вычисление локальной аппроксимации градиента $\overline{\nabla^i_t F }$ выполняется по другой схеме. Для этого сначала определим вспомогательный градиент:
\begin{equation}
    \nabla^i_t F = \overline{\nabla^i_{t-1} F } + \nabla f_i (\overline{x}_t^i) - \nabla f_i (\overline{x}_{t-1}^i).
\end{equation}

После того, как каждый узел посчитает свой вспомогательный градиент, выполняется шаг аггрегации (aggregate step):
\begin{equation}
    \overline{\nabla^i_{t} F } = \sum\limits_{j=1}^N W^t_{ij} \cdot \nabla^j_t F.
\end{equation}

Наконец, приведем описание децентрализованного алгоритма Франк-Вульфа:

\begin{algorithm}[h]
\caption{Децентрализованный метод условного градиента}\label{alg:Example}
\label{alg:decentralized_fw}
\begin{algorithmic}[1]
  \REQUIRE \rev{Начальные точки $x_0^i \in D$ ($i = 1, \ldots, N$)}, целевая функция $F$, константа гладкости $L$.
  \FOR { $t=0, 1, \ldots$}
  \STATE $\text{Консенсусный шаг:} $
  \begin{equation*}
      \overline{x}^i_t \longleftarrow \sum\limits_{j=1}^N W^t_{ij} \cdot x^j_t \text{, } \forall i \in V 
  \end{equation*}
  \STATE $\text{Шаг агрегации:} $
  \begin{equation*}
      \overline{\nabla^i_{t} F } \longleftarrow \sum\limits_{j=1}^N W^t_{ij} \cdot \nabla^j_t F \text{, } \forall i \in V 
  \end{equation*}
  \STATE $v_t^i \longleftarrow \text{argmin}_{v \in D} \langle \overline{\nabla^i_{t} F }, v \rangle$
  \STATE $\gamma_t \longleftarrow 2/(t+2)$
  \STATE $x_{t+1}^i \longleftarrow \overline{x}_t^i + \gamma_t (v_{t}^i - \overline{x}_{t}^i)$
  \ENDFOR
\end{algorithmic}
\end{algorithm}

\section{Теоретические оценки}
\subsection{Верхняя оценка скорости сходимости}\label{subsec:deterministic_changing_graph}

Как было указано выше, нам нужны величины $\overline{x_t}, \  \overline{\nabla_t F}$, чтобы следить, насколько результаты consensus step и aggregation step отличаются от среднего. Введем несколько предположений, чтобы  сделать оценку скорости сходимости алгоритма:

\begin{assumption}\label{assum:convexity_smoothness}
	Для каждого $i = 1, \ldots, N$ функция $f_i$ является выпуклой $L$-гладкой, т.е. для всяких $x, y\in D$ выполняется
	\begin{align*}
		0\leq f_i(y) - f_i(x) - \left\langle \nabla f_i(x), y - x \right\rangle\leq \frac{L}{2}\norm{y - x}_2^2.
	\end{align*}
\end{assumption}

\begin{assumption}\label{assum:bounded_arg}
    Пусть существует $(\{ \Delta p_t\}_{t\geq 1})$, \ $\forall t \geq 1$ неотрицательная последовательность такая, что $\Delta p_t \longrightarrow 0$ и
\begin{equation}
    \max_{i \in [N]} \| \overline{x}_t^{i} - \overline{x}_t \|_2 \leq \Delta p_t.
\end{equation}
\end{assumption}

\begin{assumption}\label{assum:bounded_grad}
    Пусть существует $\{ \Delta d_t\}_{t\geq 1}$, \ $\forall t \geq 1$ неотрицательная последовательность такая, что $\Delta d_t \longrightarrow 0$ и

\begin{equation}
\max_{i \in [N]} \| \overline{\nabla_t^i F} - \overline{\nabla_t F} \|_2 \leq \Delta d_t.
\end{equation}
\end{assumption}

Отметим, что для выполнения предыдущего предположения достаточно определять коммуникационную матрицу способом Metropolis Weights, описанном в \eqref{eq:metropolis_weights}.
Отсюда следует \cite{base} следующая оценка:
\begin{theorem}\label{th:decentralized_frank_wolfe}
Пусть выполнены предположения \ref{assum:mixing_matrix_sequence}, \ref{assum:convexity_smoothness}, \ref{assum:bounded_arg}, \ref{assum:bounded_grad}, размер шага равен $\gamma_t = 2/(t+1)$, а $C_p, \ C_g$ - положительные константы, такие что $\Delta p_t = C_p / t, \Delta d_t = C_g/t $. Тогда
\begin{equation}
    F(\overline{x}_t) - F(\overline{x}^*) \leq \cfrac{8\overline{\rho} (C_g + L C_p) + 2L\overline{\rho}^2}{t+1}
\end{equation}
для любых $t \geq 1$, где $\overline{x}^*$ - оптимальное решение задачи.
\end{theorem}

Таким образом, при выполнении предположений получаем линейную оценку скорости сходимости алгоритма.

Сформулируем и докажем леммы, которые гарантируют выполнение вышеизложенных предположений. \rev{Леммы приведем в общем виде для произвольного параметра $\alpha\in(0, 1]$, хотя нам понадобится только $\alpha = 1$.}
\begin{lemma}\label{lem:estimate_cp}
Пусть $t_0(\alpha)$ - наименьшее положительное целое число такое, что 
\begin{equation}\label{eq:def_t0}
    \max\limits_{t} \lambda_2(W^t) \leq \left( \cfrac{t_0(\alpha)}{t_0(\alpha)+1} \right)^{\alpha} \cdot \cfrac{1}{1 + (t_0(\alpha))^{-\alpha}}.
\end{equation}
Зададим шаг $\gamma_t = 1/t^{\alpha}$ в алгоритме Франк-Вульфа для $\alpha \in (0,1]$, тогда выполняется:
\begin{gather}
    \max_{i \in V} ||\overline{x}^i_t - \overline{x}_t||_2 \leq \Delta p_t = C_p/t^{\alpha}, \ \forall t \geq 1, \nonumber \\
    C_p = (t_0(\alpha))^{\alpha} \cdot \sqrt{N} \overline{\rho} \label{eq:def_cp}.
\end{gather}

\end{lemma}
\begin{proof}
В доказательстве будем писать $t_0 = t_0(\alpha)$. Покажем, что
\begin{equation} \sqrt{ \sum_{i=1}^N \| \overline{x}_{t}^{i} - \overline{{ x}}_{t} \|_2^2 } \leq \frac{C_{p}}{t^{\alpha}}, \, C_{p} = (t_{0})^{\alpha} \cdot \sqrt{N} \overline{\rho}.
\end{equation}
Заметим, что от $t = 1$ до $t = t_0$ неравенство выполняется, т.к. $\overline{x}_{t}^{i}, \ \overline{{ x}}_{t}$ принадлежат допустимому множеству, и его диаметр ограничен $\rho$. Для шага индукции предположим, что неравенство выполняется для $t \geq t_0$. По определению,
\begin{equation*}  x_{t+1}^{i} = (1-t^{-\alpha}) \overline{ x}_{t}^{i} + t^{-\alpha}  v_{t}^{i}.
\end{equation*}
Обозначим $\tilde{a}_t = \dfrac{1}{N} \sum\limits_{j=1}^N a_t^i $. Так как для $\overline{x}_i = \sum\limits^N_{j=1} W_{ij} \cdot x_j$ выполняется
\begin{equation*}
    \sqrt{\sum\limits_{i=1}^N || \overline{x}_i - \rev{\overline x}||^2} \leq |\lambda_2(W^t)| \cdot \sqrt{\sum\limits_{i=1}^N || x_i - \rev{\overline x}||^2},
\end{equation*}
то получаем 
\begin{align*}
	\sum_{i=1}^N &\| \overline{ x}_{t+1}^{i} - \overline{{ x}}_{t+1} \|_2^2
	\leq |\lambda_{2}( W^t)|^{2} \times \\
	&\times\sum_{j=1}^{N} \| (1-t^{-\alpha})(\overline{x}_{t}^{j} - \overline{x}_{t}) + t^{-\alpha}( v_{t}^{j} - \tilde{ v}_{t}) \|_{2}^{2} 
\end{align*}
Важно отметить, что показанное выше неравенство справедливо только для рассматриваемого шага $t$, так как на каждой итерации алгоритма может меняться $W^t$, а значит и $\lambda_2(W^t)$. С другой стороны, количество случайных графов на $N$ вершинах конечно, а значит, $\forall t \ \lambda_2(W^t) \leq \lambda = \max\limits_t \lambda_2(W^t)$.

В таком случае,
\begin{align*}
	\sum_{i=1}^N &\| \overline{ x}_{t+1}^{i} - \overline{{ x}}_{t+1} \|_2^2\\
	&\leq \lambda^{2} \sum_{j=1}^{N} \| (1-t^{-\alpha})(\overline{x}_{t}^{j} - \overline{x}_{t}) + t^{-\alpha}( v_{t}^{j} - \tilde{ v}_{t}) \|_{2}^{2} \\
	&\leq \lambda^2 \sum_{j=1}^N \Big( (1 - t^{-\alpha})^2\norm{\overline x_t^j - \overline x_t}_2^2 + \rho^2 t^{-2\alpha} \\
	&\qquad~~+ 2t^{-2\alpha}(1 - t^{-\alpha})^2 \rho\norm{\overline x_t^j - \overline x_t}_2 \Big) \\
	&\leq\lambda^2 \sum_{j=1}^N \Big( \norm{\overline x_t^j - \overline x_t}_2^2 + \rho^2 t^{-2\alpha} \\
	&\qquad~~+ 2\rho t^{-\alpha} \norm{\overline x_t^j - \overline x_t}_2 \Big) \\
	&\overset{\circledOne}{\leq} \lambda^2 t^{-2\alpha} (C_p^2 + N\rho^2) \\
	&\qquad~~+ 2\rho t^{-\alpha}\sqrt{N} \sqrt{\sum_{j=1}^N \norm{\overline x_t^j - \overline x_t}_2} \\
	&\leq\lambda^2 t^{-2\alpha} (C_p + \sqrt{N}\rho)
	\leq \cbraces{\lambda C_p \frac{(t_0)^\alpha + 1}{(t_0)^\alpha\cdot t^\alpha}}^2.
\end{align*}
Здесь в $\circledOne$ был использован тот факт, что для неотрицательных $c_1, \ldots, c_N\in\R$ выполняется 
$$
\sum\limits^{N}_{j = 1} c_j \leq \sqrt{N} \sqrt{\sum\limits^{N}_{j = 1} c_j^2}.
$$
Из (3.4) получаем, что:
\begin{equation}
    \lambda \cdot \dfrac{(t_0)^{\alpha} + 1}{(t_0)^{\alpha} \cdot t^{\alpha}} \leq \dfrac{1}{(t+1)^{\alpha}}, 
\end{equation}
    шаг индукции выполняется, тогда
\begin{equation} \sqrt{ \sum_{i=1}^N \| \overline{x}_{t}^{i} - \overline{{ x}}_{t} \|_2^2 } \leq \frac{C_{p}}{t^{\alpha}}, \, C_{p} = (t_{0})^{\alpha} \cdot \sqrt{N} \overline{\rho},
\end{equation} откуда следует доказываемое утверждение.
    
\end{proof}

Доказанная лемма гарантирует выполнение условия на скорость сходимости последовательности точек, в следующей лемме рассмотрим скорость сходимости последовательности градиентов.

Напомним обозначения:

\begin{equation*}
    \nabla^i_t F = \overline{\nabla^i_{t-1} F } + \nabla f_i (\overline{x}_t^i) - \nabla f_i (\overline{x}_{t-1}^i),
\end{equation*}
\begin{equation*}
    \overline{\nabla^i_{t} F } = \sum\limits_{j=1}^N W^t_{ij} \cdot \nabla^j_t F.
\end{equation*}

\begin{lemma}\label{lem:estimate_cg}
    Зададим шаг $\gamma_t = 1/t^{\alpha}$ в алгоритме Франк-Вульфа для $\alpha \in (0,1]$, каждая из функций $f_i$ $L-$гладкая, тогда выполняется:
\begin{gather}
     \max \| \overline{\nabla_{t}^{i} F} - \overline{\nabla_{t} F} \|_{2} \leq \frac{C_{g}}{t^{\alpha}}, \nonumber \\
     C_{g} = 2\sqrt{N}(t_{0})^{\alpha} (2C_{p} + \overline{\rho})L \label{eq:def_cg}.
\end{gather}

\end{lemma}

\begin{proof}

    Заметим, что от $t = 1$ до $t = t_0$ неравенство выполняется, что следует из ограниченности градиентов. Для шага индукции предположим, что неравенство выполняется для $t \geq t_0$.
    
    Определим вспомогательную переменную $\delta f^i_{t+1} = \nabla f_i(\overline{x}^i_{t+1} - \nabla f_i (\overline{x}^i_t)$.
    Тогда перепишем $\nabla^i_{t+1} F = \delta f^i_{t+1} + \overline{\nabla^i_t F}$, а так же $\overline{\nabla^i_{t+1} F} = \sum\limits^N_{j= 1 } W_{ij} \nabla ^j_{t+1} F$, получаем, оценив $\lambda_2(W^t)$ аналогично прошлому доказательству:
    \begin{align*} 
    \sum_{i=1}^{N} &\| \overline{\nabla_{t+1}^{i} F} - \overline{\nabla_{t+1} F} \|_{2}^{2} \\
    &\leq (\lambda_{2}( W^t))^{2} \cdot \sum_{i=1}^{N} \| \overline{\nabla_{t}^{i} F} + \delta f_{t+1}^{i} - \overline{\nabla_{t+1} F} \|_{2}^{2} \\
    &\leq \lambda \cdot \sum_{i=1}^{N} \| \overline{\nabla_{t}^{i} F} + \delta f_{t+1}^{i} - \overline{\nabla_{t+1} F} \|_{2}^{2}.
    \end{align*}
    Аналогично, определим $\delta F_{t+1} = \overline{\nabla_{t+1} F} - \overline{\nabla_t} F$, тогда правую часть (3.12) с помощью неравенства Коши-Буняковского-Шварца можно ограничить как

    \begin{align*}
        \sum_{i=1}^{N} &\| \overline{\nabla_{t}^{i} F} + \delta f_{t+1}^{i} - \overline{\nabla_{t+1} F} \|_{2}^{2} \\
        &\leq \sum_{i=1}^{N} \big(\| \overline{\nabla_{t}^{i} F} - \overline{\nabla_{t} F} \|_{2}^{2} + \| \delta f_{t+1}^{i} - \delta F_{t+1} \|_{2}^{2} \\
        &\quad+ 2 \cdot \| \delta f_{t+1}^{i} - \delta F_{t+1} \|_{2} \cdot \| \overline{\nabla_{t}^{i} F} - \overline{\nabla_{t} F} \|_{2}\big).
	\end{align*}
    Кроме того, справедливо
    \begin{align*}
        \| \delta f_{t+1}^{i} \|_{2} &= \|\nabla f_{i}(\overline{ x}_{t+1}^{i}) - \nabla f_{i}(\overline{ x}_{t}^{i}) \|_{2} \leq L \| \overline{ x}_{t+1}^{i} - \overline{ x}_{t}^{i} \|_{2}  \\
        &\leq L \Big\| \sum_{j=1}^{N} W_{ij} \big( ( x_{t+1}^{j} - \overline{ x}_{t}^{j}) + (\overline{ x}_{t}^{j} - \overline{ x}_{t}^{i}) \big) \Big\|_{2} \\
        &\leq L \sum_{j=1}^{N} W_{ij} \Big( t^{-\alpha} \rho + 2C_{p} t^{-\alpha} \Big) \\
        &= (2C_{p} + \rho) Lt^{-\alpha},
\end{align*}
где последнее неравенство записано с помощью результата Леммы 5.

Используя неравенство треугольника, оценим еще одно слагаемое из (3.14):
\begin{align*}
        \| \delta &f_{t+1}^{i} - \delta F_{t+1} \|_{2} \\
        &= \Big\| \big(1-\frac{1}{N}\big) \delta_{t+1}^{i} + \frac{1}{N} \sum_{j \neq i} \delta_{t+1}^{j} \Big\|_{2} \\
        &\leq \bigg( 1-\frac{1}{N} \bigg) \| \delta_{t+1}^{i} \|_{2} + \frac{1}{N} \sum_{j \neq i} \| \delta_{t+1}^{j} \|_{2} \\
        &\leq 2 \bigg( 1-\frac{1}{N} \bigg) (2C_{p} + \overline{\rho}) Lt^{-\alpha} \leq 2(2C_{p} + \overline{\rho}) Lt^{-\alpha}.
\end{align*}

Итого, получаем окончательную оценку (3.14):
\begin{align*}
        \sum_{i=1}^{N} &\| \overline{\nabla_{t}^{i} F} + \delta f_{t+1}^{i} - \overline{\nabla_{t+1} F} \|_{2}^{2} \\
        &\leq t^{-2\alpha} (C_{g}^{2} + 4N(2C_{p} + \rho)^{2} L^{2}) + \\
        &~~+ t^{-\alpha} 4L(2C_{p} + \rho) \sqrt{N} \sqrt{\sum\nolimits_{i=1}^{N} \| \overline{\nabla_{t}^{i} F} - \overline{\nabla_{t} F} \|_{2}^{2}} \\
        &\leq t^{-2\alpha} \cdot \big( C_{g} + 2L \sqrt{N} (2C_{p} + \rho) \big)^{2} \\
        &\leq \cbraces{\frac{(t_{0})^{\alpha} + 1}{(t_{0})^{\alpha} \cdot t^{\alpha}} \cdot C_{g}}^{2}.
\end{align*}

Взяв корень из обеих частей неравенства, получаем:

\begin{equation}
    \sqrt{\sum_{i=1}^{N} \| \overline{\nabla_{t+1}^{i} F} - \overline{\nabla_{t+1} F} \|_{2}^{2} } \leq \lambda \Big( \frac{(t_{0})^{\alpha} + 1}{(t_{0})^{\alpha} \cdot t^{\alpha}} \cdot C_{g} \Big).
\end{equation}
И, с учетом (3.8), окончательно завершаем шаг индукции, откуда следует (3.10)
\end{proof}
Таким образом, \rev{теорема~\ref{th:decentralized_frank_wolfe}} дает нам верхнюю оценку скорости сходимости алгоритма, а \rev{лемма~\ref{lem:estimate_cp}} и \rev{лемма~\ref{lem:estimate_cg}} гарантируют выполнение нужных предположений. \rev{Осталось заметить, что для выполнения условия~\eqref{eq:def_t0} достаточно взять
\begin{align*}
	t_0 = \left\lceil\frac{2}{1 - \lambda}\right\rceil.
\end{align*}
Тогда, согласно определениям \ref{eq:def_cp} и \ref{eq:def_cg}, получим
\begin{align*}
	LC_p + C_g = O(N L\overline\rho^2\chi^2).
\end{align*}
Получаем окончательную скорость сходимости.
\begin{corollary}
	Для достижения точности $\varepsilon$, т.е. для выполнения условия
	\begin{subequations}\label{eq:def_accuracy}
		\begin{align}
			\frac{1}{N} \sum_{i=1}^N f_i(x_N^i) - f(x^*)\leq \varepsilon, \\
			\max_{i\in[N]} \E\norm{\overline x_N^i - \overline x_N}_2\leq \varepsilon,
		\end{align}
	\end{subequations}
	необходимо $$N = O\cbraces{\frac{1}{(1 - \lambda)^2}\frac{L\overline\rho^2}{\varepsilon}}$$ итераций алгоритма \ref{alg:decentralized_fw}.
\end{corollary}
Также заметим, что можно применить консенсусную процедуру, в которой коммуникационная матрица $W^t$ заменится на последовательность матриц $W^{t + \tau - 1}\ldots W^t$, где $\tau = \lceil\chi\rceil$. Это позволит получить следующий результат.
\begin{corollary}
	Для достижения точности $\varepsilon$ (см.~\eqref{eq:def_accuracy})  с использованием консенсусной процедуры необходимо 
	\begin{align*}
		N_{comm} = O\cbraces{\frac{1}{1 - \lambda} \frac{L\overline\rho^2}{\varepsilon}}
	\end{align*}
	коммуникационных шагов и
	\begin{align*}
		N_{orcl} = O\cbraces{\frac{L\overline\rho^2}{\varepsilon}}
	\end{align*}
	локальных вызовов линейного минимизационного оракула на каждом узле.
\end{corollary}
}

\subsection{Случайная коммуникационная матрица}

Можно провести аналогичные рассуждения, но в случае, когда матрица коммуникации имеет случайное распределение. Доказательство будет строиться на оценке отклонения от консенсуса $\overline{x_t}$ и $\overline{\nabla_t F}$. Так как матрица является стохастической, то во все невязки будут оцениваться по матожиданию. В данном разделе мы не приводим доказательства, так как они во многом повторяют часть с неслучайной матрицей. Введем соответствующие предположения.
\begin{assumption}\label{assum:mixing_matrix_sequence_expectation}
	На каждом шаге алгоритма матрица $W^t$ является случайной и имеет распределение $\mathcal{W}$. Существует такое $\lambda < 1$, что для любого $t = 0, 1, \ldots$ выполняется
	\begin{enumerate}
		\item $W^t\mathbf{1} = \mathbf{1},~ \mathbf{1}^\top W^t = \mathbf{1}^\top$.
		\item $[W^t]_{ij} = 0$ если $(i, j)\notin E^t$.
		\item Для любого $x\in\R^n$ выполняется 
		\begin{align*}
			\E\norm{\cbraces{W^t - \frac{1}{N}\mathbf{1}\mathbf{1}^\top}x}\leq \lambda\norm{x}.
		\end{align*}
	\end{enumerate}
\end{assumption}

\begin{assumption}\label{assum:bounded_arg_expectation}
    Пусть существует $(\{ \Delta p_t\}_{t\geq 1})$, \ $\forall t \geq 1$ неотрицательная последовательность такая, что $\Delta p_t \longrightarrow 0$, тогда
\begin{equation}
    \max_{i \in [N]} \E\| \overline{x}_t^{i} - \overline{x}_t \|_2 \leq \Delta p_t
\end{equation}
\end{assumption}

\begin{assumption}\label{assum:bounded_grad_expectation}
    Пусть существует $\{ \Delta d_t\}_{t\geq 1}$, \ $\forall t \geq 1$ неотрицательная последовательность такая, что $\Delta d_t \longrightarrow 0$, тогда
\begin{equation}
\max_{i \in [N]} \E\| \overline{\nabla_t^i F} - \overline{\nabla_t F} \|_2 \leq \Delta d_t
\end{equation}
\end{assumption}

Сформулируем основной результат для случайной матрицы коммуникаций.
\begin{theorem}
Пусть выполняются предположения \ref{assum:convexity_smoothness}, \ref{assum:mixing_matrix_sequence_expectation}, \ref{assum:bounded_arg_expectation}, \ref{assum:bounded_grad_expectation} размер шага равен $\gamma_t = 2/(t+1)$, а также каждая из функций $f_i$ выпуклая и $L$-гладкая. Пусть $C_p, \ C_g$ - положительные константы, такие что $\Delta p_t = C_p / t, \Delta d_t = C_g/t $. Тогда
\begin{equation}
    \E F(\overline{x}_t) - F(\overline{x}^*) \leq \cfrac{8\overline{\rho} (C_g + L C_p) + 2L\overline{\rho}^2}{t+1}
\end{equation}

для любых $t \geq 1$, где $\overline{x}^*$ - оптимальное решение задачи.
\end{theorem}

Таким образом, при выполнении предположений получаем линейную оценку скорости сходимости алгоритма.

Сформулируем леммы, которые гарантируют выполнение вышеизложенных предположений. Доказательство лемм аналогично доказательствам в разделе \ref{subsec:deterministic_changing_graph}.
\begin{lemma}
Пусть $t_0$ -- наименьшее положительное целое число такое, что 
\begin{equation}
    \lambda \leq \left( \cfrac{t_0(\alpha)}{t_0(\alpha)+1} \right)^{\alpha} \cdot \cfrac{1}{1 + (t_0(\alpha))^{-\alpha}}.
\end{equation}
Зададим шаг $\gamma_t = 1/t^{\alpha}$ в алгоритме Франк-Вульфа для $\alpha \in (0,1]$, тогда выполняется:
\begin{gather}
    \max_{i \in V} \E||\overline{x}^i_t - \overline{x}_t||_2 \leq \Delta p_t = C_p/t^{\alpha}, \ \forall t \geq 1 \\
    C_p = (t_0(\alpha))^{\alpha} \cdot \sqrt{N} \overline{\rho}
\end{gather}

\end{lemma}

\begin{lemma}
    Зададим шаг $\gamma_t = 1/t^{\alpha}$ в алгоритме Франк-Вульфа для $\alpha \in (0,1]$, каждая из функций $f_i$ $L-$гладкая, тогда выполняется:
\begin{gather}
     \max_{i\in V} \E\| \overline{\nabla_{t}^{i} F} - \overline{\nabla_{t} F} \|_{2} \leq \frac{C_{g}}{t^{\alpha}}, \\ C_{g} = 2\sqrt{N}(t_{0})^{\alpha} (2C_{p} + \overline{\rho})L
\end{gather}
\end{lemma}

\rev{
Аналогично случаю с детерминированно изменяющейся матрицей, подытожим результаты.
\begin{corollary}
	Для достижения точности $\varepsilon$, т.е.
	\begin{subequations}\label{eq:def_accuracy_stoch}
		\begin{align}
			\E\sbraces{\frac{1}{N} \sum_{i=1}^N f_i(x_N^i)} - f(x^*)\leq \varepsilon, \\
			\max_{i\in[N]} \E\norm{\overline x_N^i - \overline x_N}_2\leq \sqrt\varepsilon,
		\end{align}
	\end{subequations}
\end{corollary}
\noindent необходимо $$N = O\cbraces{\frac{1}{(1 - \lambda)^2} \frac{L\overline\rho^2}{\varepsilon}}$$ итераций. При использовании консенсусной процедуры необходимо
	\begin{align*}
	N_{comm} = O\cbraces{\frac{1}{1 - \lambda} \frac{L\overline\rho^2}{\varepsilon}}
\end{align*}
коммуникационных шагов и
\begin{align*}
	N_{orcl} = O\cbraces{\frac{L\overline\rho^2}{\varepsilon}}
\end{align*}
локальных вызовов линейного минимизационного оракула на каждом узле.
}

\section{Численные эксперименты}

Задача лассо-регрессии (Least Absolute Shrinkage and Selection Operator) - разновидность задачи линейной регрессии, метод регуляризации линейной модели (L1 регуляризация). В LASSO к целевой функции добавляется штраф на сумму абсолютных значений параметров модели, что приводим к тому, что коэффициенты признаком с наименьшей информативностью уменьшаются, или, часто, приравиваются к нулю, что отличает метод L1 регуляризации от классической L2 регуляризации.

\begin{figure}[ht]
	\centering
	\includegraphics[width=6cm]{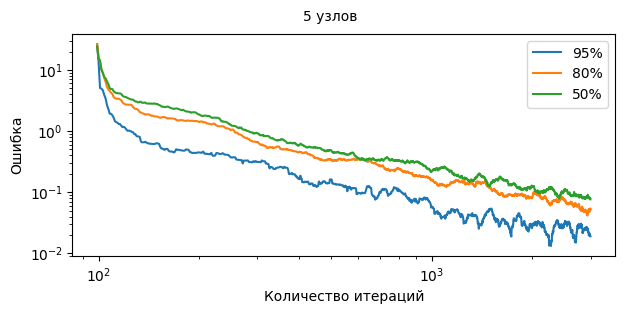}
	\includegraphics[width=6cm]{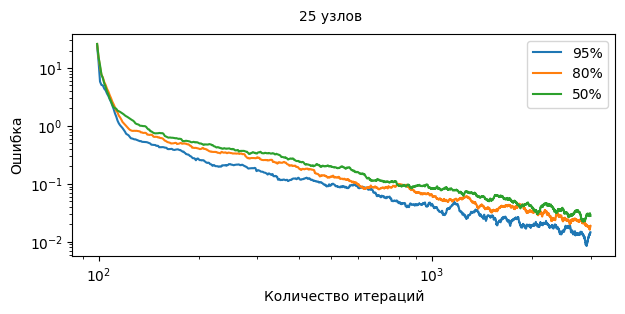}
	\includegraphics[width=6cm]{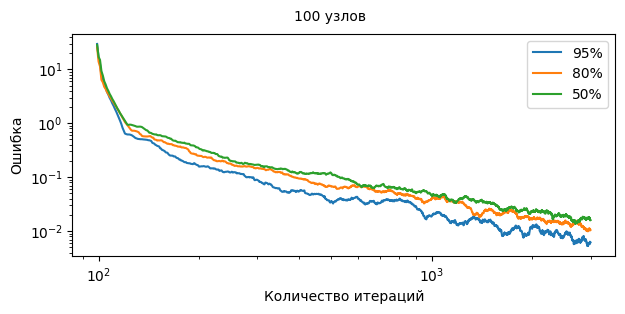}
	
	\caption{Синей кривой обозначен график для $p=0.95$, оранжевым - для $p=0.8$, и зеленым - для $p=0.5$. Заметим, что для всех значений вероятностей выполняется $p > \log N / N$, что почти наверное гарантирует связность графа \cite{raigor}.}
	\label{Fig1}
\end{figure}

Использование регуляризации позволяет произвести отбор наиболее важных признаков модели, сделав ее более интерпретируемой.

Ставится задача так: есть выборка из $n$ наблюдений переменной, значения наблюдений записаны в векторе $y$, в матрице $A$ - значения признаков, в векторе $x$ - параметры (веса) модели, пусть $\theta$ - штраф на сложность модели.

Тогда задача выглядит следующим образом:
\begin{equation}
    F_{LASSO}(x) = \dfrac{1}{2} \|y - Ax\|_2^2 + \theta \|x\|_1 \longrightarrow \min
\end{equation}
Такая постановка задачи эквивалентна задаче минимизации квадратичного функционала на симплексе, что упрощает ее решение.
\begin{gather*}
        F_{LASSO}(x) = \dfrac{1}{2} \|y - Ax\|_2^2  \longrightarrow \min \\
        \text{s.t. } \|x\|_1 \leq t.
\end{gather*}
Таким образом, действие оракула алгоритма Франк-Вульфа для задачи LASSO можно описать таким образом: алгоритм считает градиент в точке и двигается в сторону, противоположному направлению наибольшей по модулю компоненты градиента.

Опишем гипотезы, которые были проверены при моделировании алгоритма.
Во-первых, проверялась зависимость скорости сходимости алгоритма при степенях разреженности графа. Чтобы это сделать, мы воспользовались моделью Эрдеша-Реньи генерации случайных графов, которая заключается в том, что каждое возможное ребро графа генерируется с вероятностью $p$. 

Итак, для разных значений $p$ (а значит и для разной разреженности графов) была измерена скорость сходимости алгоритма для одной и той же задачи для $N=5, N=25, N=100$ (рис. \ref{Fig1}, \ref{Fig2}).


Кроме того, была смоделирована ситуация, когда граф, который инициализируется изначально, не генерируется заново на каждой итерации, а меняется слабо, причем количество ребер в нем остается неизменным. 

\begin{figure}[ht]
    \centering
    \includegraphics[width=6cm]{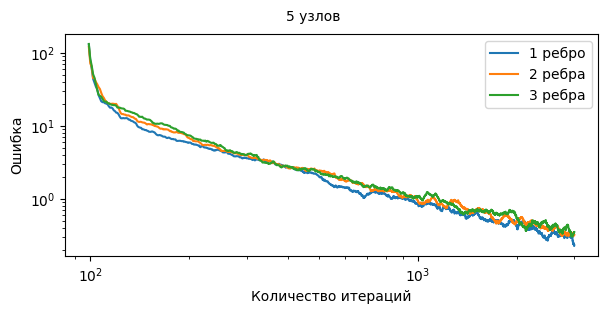}
    \includegraphics[width=6cm]{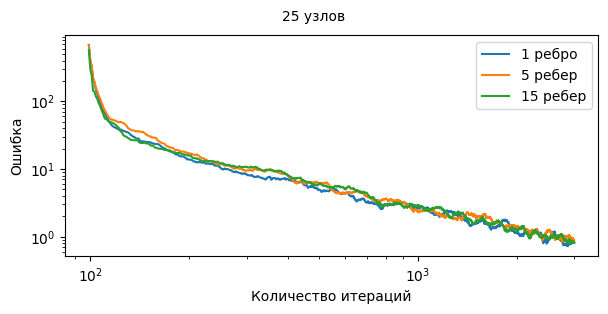}

    \caption{Изначальный граф сгенерирован с вероятностью появления ребер $p = 0.5$}
    \label{Fig2}
\end{figure}

Как видно из графиков, не удается обнаружить зависимости скорости сходимости от количества замененных ребер.

\section{Заключение}

В данной работе был рассмотрен метод Франк-Вульфа на переменных во времени графах. С теоретической точки зрения, было рассмотрено два режима изменения графа: детерминированная и стохастическая последовательность графов. Для обоих случаев показано, что алгоритм сходится со скоростью порядка $O(1/t)$, где $t$ -- номер итерации. Также были проведены численные эксперименты, подтверждающие теоретические результаты.

Данная работа поддержана грантом Российского Научного Фонда (проект No. 23-11-00229), \url{https://rscf.ru/en/project/23-11-00229/}.

\end{document}